\documentclass[11pt]{article}

\usepackage{amsmath}
\usepackage{amsthm}
\usepackage{amsfonts}

\newtheorem{thm}{Theorem}[section]
\newtheorem{lemma}[thm]{Lemma}
\newtheorem{example}[thm]{Example}

\newtheorem{assertion}[thm]{Proposition}
\theoremstyle{definition}

\theoremstyle{remark}

\newcommand{\lin}{\lim_{n\to\infty}}
\newcommand{\liti}{\lim_{t\to\infty}}

\newcommand{\mmp}{\mathbb{P}}

\newcommand{\me}{\mathbb{E}}
\newcommand{\reals}{\mathbb{R}}
\newcommand{\mr}{\mathbb{R}}
\newcommand{\naturals}{\mathbb{N}}
\newcommand{\naturalso}{\mathbb{N}_0}

\author{
Alexander Gnedin\footnote{Department of Mathematics, Utrecht University, Postbus 80010, 3508 TA Utrecht, The Netherlands,
\;A.V.Gnedin@uu.nl},
Alexander Iksanov\footnote{Faculty of Cybernetics, National T. Shevchenko University of Kiev, 01033 Kiev, Ukraine,\; iksan@unicyb.kiev.ua},
Alexander Marynych\footnote{Faculty of Cybernetics, National T. Shevchenko University of Kiev, 01033 Kiev, Ukraine,\;marynych@unicyb.kiev.ua}
}
\title{The Bernoulli sieve: an overview}
\begin{document}
\maketitle
\begin{abstract}
The Bernoulli sieve is a version of the classical
balls-in-boxes
occupancy scheme, in which random frequencies of infinitely many
boxes are produced by a multiplicative random walk, also known as
the residual allocation model or stick-breaking.
We give an overview of  the limit theorems concerning
the number  of boxes occupied by some balls out of the first $n$ balls thrown, and present some
new results concerning  the number of empty boxes within the occupancy range.
\end{abstract}

\section{Introduction}
In a classical occupancy scheme $n$ balls are thrown
independently in an infinite array of boxes with probability $p_k$
of hitting box $k=1,2,\dots$, where $(p_k)_{k\in\naturals}$ is a
fixed sequence of positive frequencies summing up to one.
The quantities of traditional interest are
\begin{itemize}
\item $K_n$ the number of boxes occupied by at least one of $n$
balls,
\item $K_{n,r}$ the number of boxes occupied by exactly $r$
out of $n$ balls,
\item $M_n$  the range of occupancy, equal to the maximal index of occupied box,
\item
$L_n:=M_n-K_n$ the number of empty boxes within the occupancy range,
\item $Z_n$ the number of balls in the $M_n$th box.
\end{itemize}
In applications `boxes' are clusters, species, types of data etc.
The  quantities in the list characterise the sample variability, which for large $n$ is dominantly determined by
the boxes occupied by a few balls, thus  determined by the way  the frequencies $p_k$ approach zero as $k\to\infty$.
The first two  variables are functionals of the induced partition of $n$,
defined as the unordered collection
of  positive occupancy counts.

The {\it Bernoulli sieve} is a version of the occupancy scheme
with random frequencies
\begin{equation}\label{freq_def}
p_k:=W_1W_2\cdots W_{k-1}(1-W_k), \ \ k\in\naturals,
\end{equation}
where $(W_k)_{k\in\naturals}$ are independent copies of a random
variable $W$ taking values in $(0,1)$.
The name derives from the following recursive construction based on i.i.d. $q_k=_d1-W$:
at round 1 a coin with probability $q_1$ for heads is flipped for each of $n$ balls and every time it turns heads
the ball is put in box 1, then at round 2 a coin with probability $q_2$ for heads is flipped
for each of the remaining balls
and every time it turns heads
the ball  is sent to box 2,
and so on until all balls are allocated in boxes.

It is useful to identify
frequencies {\rm (\ref{freq_def})} with the lengths of
component intervals induced by splitting $[0,1]$
at points visited by a multiplicative random walk
$(Q_k)_{k\in\naturals_0}$, where
$$Q_0:=1, \ \ Q_j:=\prod_{i=1}^j W_i, \ \ j\in\naturals.$$
In the spirit of Kingman's `paintbox representation' of exchangeable partitions \cite{GnePit}, we
may identify the boxes with open intervals
$(Q_k,Q_{k-1})$, and
mark the balls by independent points $U_1,\ldots,U_n$
sampled from the uniform $[0,1]$ distribution,
independently of $(Q_k)$.
The event $U_i\in (Q_{k-1},Q_k)$ then means that ball $i$ falls in box $k$.
Keep in mind that in the natural order the intervals are indexed from the right to the left,
thus the occupancy range is determined by the interval containing the leftmost mark $\min(U_1,\dots,U_n)$.


The Bernoulli sieve has nonrandom frequencies only when the  law of $W$ is a Dirac mass $\delta_p$ located at
some $p\in (0,1)$, the frequencies $p_k$ comprise then a geometric distribution.
Results for this case can be readily recast from the numerous studies on
sampling from the geometric distribution \cite{BrussGrubel,BrussOC,GohHitczenko,KirschProd} and related models like the
leader election algorithms \cite{BaryshEisenStengle,FMS_LEA,JS_LEA,ProdHTSL}, absorption sampling \cite{AbsSamplingDunkl,AbsSamplingKemp} etc.
It is known that asymptotic expansions of the moments  of $K_n, M_n$ and many other quantities
have a component that oscillates periodically on the $\log n$-scale with a small amplitude
\cite{FMS_LEA,LouchProd}.
The same applies to distributions of the $L_n$'s \cite{GohHitczenko,LouchProdGaps}.
There  are some peculiarities in the symmetric case $p=1/2$  \cite{FMS_LEA,ProdHTSL}.

The best analytically tractable case involves random factors having {\rm beta}$(\theta,1)$ density
${\mathbb P}\{W\in {\rm d}x\}=\theta x^{\theta-1}{\rm d}x$ on  $(0,1)$  with parameter $\theta>0$.
In this case the Bernoulli sieve may be viewed as a way to generate
a  random partition of $n$
which follows the multivariate distribution known as
the Ewens sampling formula \cite{ABT}.
This model  has been widely studied in connection with problems
of combinatorics, statistics and biology.
In particular,
the case $\theta=1$ of uniform factors is related to records and cycle patterns of random
permutations under the uniform distribution on the symmetric group.
It is well known  \cite{ABT} that
 $(K_n-\theta\log n)/(\theta\log n)^{1/2}$ is asymptotically normal,  and that the $K_{n,r}$'s converge jointly
to independent Poisson$(\theta/r)$ random variables.
These classical  results  are  complemented by the observation that $M_n$ exhibits the same
asymptotics of moments and distribution as $K_n$, and
the number of empty boxes
has the following surprising limit law:

\begin{thm}\label{emptybox}
{\rm \cite{GINR}}
If~\,$W$ has ${\rm beta}(\theta, 1)$ distribution then $L_n\to_d L_\infty$, where
$L_\infty$ has probability generating function
$$\me s^{L_\infty}={\Gamma(1+\theta)\Gamma(1+\theta-\theta s)\over
\Gamma(1+2\theta-\theta s)}, ~~~s\in [0,1],$$
which corresponds to a mixed Poisson distribution
with the parameter distributed like $\theta\,|\log (1-W)|$.
\end{thm}


Throughout we shall use the following notation for the moments
$$\mu:=\me |\log W|,\;\;\;\sigma^2:={\rm Var}\,(\log W),\;\;\;\nu:=\me |\log (1-W)|,$$
which may be finite or infinite. The {\it standing assumption} for
what follows is that the distribution of $|\log W|$ is non-lattice.
In particular, the case of sampling from the geometric distribution will be excluded.



\section{Markov chains and distributional recursions}

A random combinatorial structure which captures the occupancy of
boxes by $n$ indistinguishable balls
is the {\it weak composition} $C^*_n$
comprised of nonnegative integer parts
summing up to $n$.
The term {\it weak} composition means that zero parts are allowed, for
instance, the sequence $(2, 3, 0, 1, 0, 0, 1, 0, 0, 0, \ldots)$
(padded by infinitely many $0$'s) is a possible value of $C^*_7$.
A related structure which contains less information is a composition
$C_n$  obtained by discarding zero parts of $C^*_n$. Discarding further the
order of parts in $C_n$ yields a random partition of $n$.
The parts of $C^*_n$ can be represented
(see \cite[p. 452]{GnePit}) as the magnitudes  of jumps of a
time-homo\-geneous non\-increa\-sing Markov chain $Q^*_n=(Q^*_n(k))_{k\in\naturalso}$ on integers, which starts at $n$
and moves from $n$ to $m$ with transition probabilities
$$q^*(n,m)={n \choose m}\me (1-W)^{n-m}W^m,\;\;m=0,\ldots,n.$$
In the same direction, parts of the
composition $C_n$ are the magnitudes of jumps of a Markov chain
$Q_n=(Q_n(k))_{k\in\naturalso}$  with
transition probabilities
$$q(n,m)={n \choose m}\frac{\me (1-W)^{n-m}W^m}{1-\me W^n},\;\;m=0,\ldots,n-1.$$
This Markovian realisation implies
the following distributional recursions (see \cite[Section 3]{GINR}):
$$M_0=0,\;\;M_n=_d M_{Q^\ast_n(1)}+1, \ \ n\in\naturals,$$
\begin{equation}\label{kn}
K_0=0,\;\;K_n=_d K_{Q_n(1)}+1, \ \ n\in\naturals,
\end{equation}
\begin{equation}\label{empty}
L_0=0,\;\;L_n=_d
L_{Q^\ast_n(1)}+1_{\{Q^\ast_n(1)=n\}},n\in\naturals,
\end{equation}
where in the right-hand side $Q^\ast_n(1)$ is assumed independent of $\{M_n:n\in\naturals\}$ and
$\{L_n:n\in\naturals\}$, and $Q_n(1)$  independent of
$\{K_n:n\in\naturals\}$.
Analysis of the recursions by known direct methods is difficult, as these impose
restrictive conditions on the moments of $Q_n(1)$ or $Q_n^*(1)$.
Nevertheless, coupling with the multiplicative random walk allows to gain a lot of information about the compositions.
For instance,
let $g(n,m)$ be the potential function, equal to the probability that $Q_n$ ever visits  state
$m$,

$$g(n,m) =
\sum_{j=0}^{\infty}\mmp\{Q_n(j)=m\}.$$
The coupling implies that (\cite[Proposition 5]{Gne})
\begin{equation}\label{g_conv}
\lim_{n\to\infty}g(n,m)=\frac{1-\me W^m}{\mu m},
\end{equation}
which is  $0$ if $\mu=\infty$.

The coupling readily implies stochastic subadditivity $M_{n+m}<_d M_n+M'_m$ where the terms in the right-hand side are
independent. Indeed, note first that $M_n$ is nondecreasing. Now, when $n$ balls have been allocated within the range
$M_n$, adding $m$ new balls leads to (stochastically) maximal increase of the occupancy range when all $m$ fall outside
the old range $M_n$, in which event the new range of occupancy is distributed like $M_n+M'_m$.
With analogous notation, $L_{n+m}<_d L_n+L'_m$ for exactly the same reason (although $L_n$ is not monotone).

\section{Asymptotics of $M_n$}

Passing from the multiplicative to conventional (additive) random walk we introduce
\begin{equation}\label{add_rand_walk_def}
S_0:=0, \ \ S_k:=|\log W_1|+\ldots+|\log W_k|, \ \
k\in\mathbb{N}.
\end{equation}
In this scenario the Bernoulli sieve can be defined as allocation of balls with exponentially distributed marks
$E_j=-\log U_j, ~1\leq j\leq n,$ in boxes $(S_k, S_{k+1}),~k\in\naturals_0$.
Define
\begin{equation}\label{N_def}
N_t:=\inf\{k\geq 1 : S_k> t\}, \ \ t\geq 0,
\end{equation}
which is the first  time  $(S_k)$ enters $(t,\infty)$.
From the extreme-value theory we know that the maximum statistic
$T_n:=\max (E_1,\ldots, E_n))$
satisfies
$T_n-\log n \to_{d} T,$
where $T$ has the standard Gumbel distribution
$\mathbb{P}\{T\leq x\}= \exp(-e^{-x})$, $x\in\mathbb{R}$.
A key observation is that
$$M_n=N_{T_n},$$
thus the asymptotic  behaviour of $M_n$ is very much the same as that
of $N_{\log n}$, and the latter can be concluded by means of  the renewal theory.
A complete description of possible limit laws and scaling/centering
constants for the number of renewals $N_t$
 \cite[Proposition A.1]{GINR} leads to the following classification of possible limit laws for $M_n$.

\begin{thm}
\label{main2}{\rm \cite{GINR}}
The following assertions are equivalent:
\begin{itemize}
\item[\rm (i)] There exist sequences $\{a_n, b_n : n\in\naturals\}$ with $a_n > 0$ and $b_n\in\reals$ such
that, as $n\to\infty$, the variable $(M_n-b_n)/a_n$ converges weakly to some non-degenerate and proper distribution.
\item[\rm (ii)] The distribution of $|\log W|$ either belongs to the domain of attraction of a stable law,
 or the function $\mmp\{|\log W|>x\}$ slowly varies
at $\infty$.
\end{itemize}
Accordingly, there are five possible modes of convergence:
\begin{enumerate}
\item[\rm (a)]
If $\sigma^2<\infty$ then, with constants $b_n=\mu^{-1}\log n$ and $a_n=(\mu^{-3}\sigma^2\log n)^{1/2}$, the limiting
distribution of $(M_n-b_n)/a_n$ is standard normal.
\item[\rm (b)]
If $\sigma^2=\infty$, and
$$\int_0^x y^2\, \mmp\{|\log W|\in {\rm d}y\} \ \sim \ L(x) \ \ x\to \infty,$$ for some function $L$ slowly varying at $\infty$,
then, with $b_n=\mu^{-1}\log n$ and $a_n=\mu^{-3/2}c_{[\log
n]}$, where $c(x)$ is any positive function satisfying
$\lim_{x\to\infty}\,xL(c(x))/c^2(x)=1$, the limiting distribution of  $(M_n-b_n)/a_n$ is standard
normal.
\item[\rm (c)]
If
\begin{equation}\label{domain1}
\mmp\{|\log W|>x \} \ \sim \ x^{-\alpha}L(x), \ \ x\to \infty,
\end{equation}
for some $L$ slowly varying at $\infty$ and $\alpha\in (1,2)$
then, with $b_n=\mu^{-1}\log n$ and
$a_n=\mu^{-(\alpha+1)/\alpha}c_{\log n}$, where $c(x)$ is any
positive function satisfying $\lim_{x\to\infty}\,xL(c(x))/c^{\alpha}(x)=1$, the
limiting distribution of  $(M_n-b_n)/a_n$ is $\alpha$-stable with characteristic
function
$$t\mapsto \exp\{-|t|^\alpha
\Gamma(1-\alpha)(\cos(\pi\alpha/2)+i\sin(\pi\alpha/2)\, {\rm
sgn}(t))\}, \ t\in\reals.$$

\item[{\rm (d)}] Assume that the
relation {\rm (\ref{domain1})} holds with $\alpha=1$. Let $r:\reals
\to\reals$ be any nondecreasing function such that
$\lim_{x\to\infty} x\mmp\{|\log W|>r(x)\}=1$ and set
$$m(x):=\int_0^x \mmp\{|\log W|>y\}{\rm d}y, \ \ x>0.$$ Then, with
$b_n=\log n / (m(\log n/ r (m(\log n))))$ and $$a_n:={r(\log
n/m(\log n))\over m(\log n)},$$ the limiting distribution of
$(M_n-b_n)/a_n$ is $1$-stable with characteristic function
$$
t\mapsto \exp\{-|t|(\pi/2-i\log|t|\,{\rm sgn}(t))\}, \ t\in\reals.
$$

\item[{\rm (e)}] If the relation {\rm (\ref{domain1})} holds for $\alpha \in [0,1)$
then, with $b_n\equiv 0$ and $a_n:=\log^\alpha n/L(\log n)$, the
limiting distribution of  $M_n/a_n$ is the Mittag-Leffler law $\theta_\alpha$ with moments

$$\int_0^\infty x^k \ \theta_\alpha({\rm d}x)\ = \ {k!\over
\Gamma^k(1-\alpha)\Gamma(1+\alpha k)}, \ \ k\in\naturals.$$
\end{enumerate}
\end{thm}

\section{Asymptotics of $K_n$}

Loosely speaking, $\nu$ controls the mean number of empty boxes,
so that
$\nu<\infty$ implies   $\lim_{n\to\infty}\me L_n={\nu/\mu}<\infty$ (Theorem \ref{L_n_mean} to follow).
Thus when $\nu<\infty$ the identity $K_n=M_n-L_n$ suggests that $K_n$ does not differ much from $M_n$.
A first result of this kind was obtained in \cite{Gne}: assuming  $\nu<\infty$ and
$\sigma^2<\infty$ it was shown that
$$(K_n-\mu^{-1}\log n)/\sqrt{\sigma^2\mu^{-3}\log n}\to_d
{\rm normal}\,(0,1), \ \ n\to\infty.$$
The proof was based on a careful analysis of the recursion {\rm (\ref{kn})} to conclude on the asymptotics of ${\rm Var}\,K_n$
and to eventually prove the normal limit.

The similarity between $M_n$ and $K_n$ was justified in full generality in \cite{GINR}, where is was shown that
under the assumption $\nu<\infty$
Theorem \ref{main2} remains valid if
$M_n$ is replaced by $K_n$.

Another approach which allows one to treat the cases of finite and
infinite $\nu$ in a unified way was proposed in \cite{GIM}.
It was suggested to approximate $K_n$ by
$N^\ast(\log n)$, where
\begin{eqnarray*}\label{nx}
N^\ast(x)&:=&\#\{k\in\naturals: p_k\geq
e^{-x}\}\\&=&\#\{k\in\naturals: W_1\cdots W_{k-1}(1-W_k)\geq
e^{-x}\},~~~  x>0.
\end{eqnarray*}
The connection exemplifies the general idea that
the variability of  $K_n$ stems from randomness in frequencies $(p_k)$ superposed with randomness in sampling,
and the first often plays a dominating role through the conditional law of large numbers
$K_n\sim{\mathbb E}(K_n\,|\,(p_k))$ a.s. (see \cite{Karlin}).
Thus we believe that the approach based on $N^\ast(x)$
offers a natural and the most adequate way to
study the asymptotics of $K_n$. The following result was proved in
\cite{GIM}.

%
%
\begin{thm}\label{main}
If there exist functions $f: \mr_+\to\mr_+$ and $g:\mr_+\to\mr$
such that $(N_t-g(t))/f(t)$ converges weakly (as
$t\to\infty$) to some non-degenerate and proper distribution, then
also $(K_n-b_n)/a_n$ converge weakly (as  $n\to\infty$) to the
same distribution, where the constants are given by
$$b_n=\int_0^{\log n} g(\log
n-y)\,\mmp\{|\log (1-W)|\in {\rm d}y\}, ~~~a_n=f(\log n).$$
\end{thm}
As in
\cite{GINR},
the convergence criterion for $N_t$ leads to a
complete characterisation of possible normalisations and limiting laws  for $K_n$, see Corollary 1.1 in \cite{GIM}.
But  Theorem \ref{main} says more: if $\nu=\infty$ the
behaviour of $L_n$ may affect the asymptotics of $K_n=M_n-L_n$. The
following example illustrates the phenomenon.
\begin{example}{\rm
Assume that, for some $\gamma\in (0,1/2)$,
$$\mmp\{W>x\}={1 \over 1+|\log(1-x)|^\gamma}, \ \  x\in [0,1).$$ Then

$$\me \log^2 W<\infty \ \ {\rm and} \ \ \mmp\{|\log(1-W)|>x\}\sim
x^{-\gamma}~~~{\rm as}~x\to\infty,$$ and in this case, $$a_n= {\rm
const}\log^{1/2}n \ \ {\rm and} \ \ d_n=\mu^{-1}(\log
n-(1-\gamma)^{-1}\log^{1-\gamma} n+o(\log^{1-\gamma}n)).$$ Thus we
see that the second term $d_n-\mu^{-1}\log n$ of centering cannot
be ignored. Moreover, one can check that
$$\me L_n\sim {1\over \mu} \sum_{k=1}^n{\me W^k\over k}\
\sim \ b_n-\mu^{-1}\log n \ \sim \ {1\over \mu (1-\gamma)}
\log^{1-\gamma}n,$$ which reveals the indispensable contribution of
$L_n$. }
\end{example}

\section{Weak convergence of $K_{n,r}$}
Assume  $\mu<\infty$.
For $B:=\{\prod_{i=1}^k W_i : k\in\naturalso\}$  the set of sites visited by the multiplicative random walk,
consider  a point process
with unit atoms located  at points of  $-\log B$
(which are the sites visited by $S_k,~k\in\naturals_0$).
By the renewal theorem the point process $-\log B-\log n$ vaguely converges
to a shift-invariant renewal process $\mathcal{P}$ on the whole line. Therefore, the point process
$nB$ converges vaguely to a point process $\mathcal{B}:=\exp(-\mathcal{P})$ on ${\mathbb R}_+$.
Think of intervals between consequitive points of $\mathcal{B}$ as a series of boxes.
Note that the process is self-similar, meaning that $c\mathcal{B}=_d \mathcal{B}$ for every $c>0$,
and has the intensity
measure $(\mu x)^{-1}{\rm d}x$, so the atoms accumulate at $0$ and $\infty$.
In the role of balls assume the points of
a  unit Poisson process $\mathcal{U}$ independent of $\mathcal{B}$.
A well-known fact of extreme value theory is that $\mathcal{U}$ is the
vague limit of the point process with unit atoms located at $nU_j,~1\leq j\leq n$.
The location of the leftmost atom of $\mathcal{U}$, say $Y$, has exponential distribution.
For $r\geq 0$ define $\hat{K}_r$
to be the number of component intervals of $(Y,\infty)\setminus\mathcal{B}$
 that contain exactly $r$ atoms of  $\mathcal{U}$. The existence of weak limits for the
occupancy counts is read off from the convergence of point processes:
\begin{thm}{\rm \cite{GIR}}
As $n\to\infty$ we have the joint convergence in distribution
$$(L_n,K_{n,1},K_{n,2},\ldots)\to_d (\hat{K}_0,\hat{K}_1,\hat{K}_2,\ldots)
$$
along with
$$
\me K_{n,r}\to \me \hat{K}_r=\frac{1}{r\mu},\;r>0.
$$
\end{thm}
When $W=_d{\rm beta}(\theta,1)$ the process $\cal B$ is Poisson with intensity $\theta x^{-1}{\rm d}x$.
By self-similarity, the partition induced by allocation of $n$ leftmost atoms of $\cal U$ is the
Ewens partition.
 The theorem allows to re-prove the results on asymptotics of the Ewens partition mentioned in Introduction,
along with  Theorem \ref{emptybox}.
Except the ${\rm beta}(\theta,1)$ case no explicit formulas for the distribution of occupancy counts are known;
in general the $\hat{K}_r$'s are neither independent, nor Poisson.
See more on self-similar partitions in \cite[Section 5]{RegCombStr}.

\section{Asymptotics of $Z_n$}

The variable $Z_n$ is analogous to the number of winners in the leader
election algorithm \cite{BSW,BrussGrubel,BrussOC,KirschProd}.

\begin{thm}{\rm \cite{GINR}} The number of balls in the last occupied box satisfies:
\begin{itemize}
\item[{\rm (1)}]
If $\mu<\infty$ then $Z_n\to_d Z$, $n\to\infty$, where the variable $Z$ has
distribution
$$
\mmp\{Z=k\}=\frac{\me (1-W)^k}{\mu k},\;\;k\in\naturals.
$$
\item[{\rm (2)}]
If {\rm (\ref{domain1})} holds with  $\alpha\in [0,1)$ then
$$
\frac{\log Z_n}{\log n}\to_d Z^{(\alpha)}, \ \ n\to\infty,
$$
where the law of $Z^{(0)}$ is $\delta_1$, while for $\alpha\in
(0,1)$ we have $Z^{(\alpha)}=_d{\rm beta}(1-\alpha,\alpha)$.
\item[{\rm (3)}]
If {\rm (\ref{domain1})} holds with $\alpha=1$ and $\mu=\infty$, then
$$
\frac{m(\log Z_n)}{m(\log n)}\to_d Z^{(1)}, \ \ n\to\infty,
$$
where $m(x)=\int_0^x \mmp\{|\log W|>y\}{\rm d}y\,$, and $Z^{(1)}=_d{\rm uniform}[0,1]$.
\end{itemize}
\end{thm}

The case $\mu<\infty$ is quite elementary, as is seen from

\begin{equation}\label{17}
\mmp\{Z_n=m\}=g(n,m)\mmp\{Q_m(1)=0\}=g(n,m)\frac{\me
(1-W)^m}{1-\me W^m}
\end{equation}
and {\rm (\ref{g_conv})}. In the case $\mu=\infty$ the result follows
from the known limit distribution of the undershoot
$U(z)=z-S_{N(z)-1}$ (see \cite{Dynkin, Erick}) and the
representation
$$
\mmp\{Z_n>k\}=\mmp\{U(E_{n,n})>E_{n,n}-E_{n-k,n}\},\;\;k\in\naturals,
$$
where $E_{1,n}\leq \ldots\leq  E_{n,n}=T_n$ are the order statistics of the
exponential variables  $E_j, ~1\leq j\leq n$.

\section{Asymptotics of $L_n$}

Although there is an explicit formula
\begin{equation}\label{empty_meam_expl}
\me L_n=\sum_{k=1}^{n}(-1)^{k+1}{ n \choose k}\frac{1-\me
(1-W)^k}{1-\me W^k},
\end{equation}
it does not seem possible to employ it in order to conclude on the asymptotic behaviour of $\me L_n$
without restrictive additional assumptions.

Using a different approach we arrived at
\begin{thm}\label{L_n_mean}
The expectation $\mathbb{E}L_n$ exhibits the following
asymptotic behaviour:
\begin{enumerate}
\item[\rm (i)] If  $\mu=\infty$ and $\nu=\infty$ then
$$\liminf_{n\to\infty}{\me W^n\over \me (1-W)^n}\leq\liminf_{n\to\infty}{\mathbb{E}L_n}\leq \limsup_{n\to\infty}{\mathbb{E}L_n}\leq
\limsup_{n\to\infty}{\me W^n\over \me (1-W)^n}.$$
In particular,
$$\lim_{n\to\infty} {\me W^n\over \me (1-W)^n}=\gamma_0\in [0,\infty]$$
implies
$\lim_{n\to\infty} \me L_n=\gamma_0$.
\item[\rm (ii)] If $\nu<\infty$ and $\mu\leq
\infty$ then
$$\lim_{n\to\infty} \me L_n=\nu/\mu.$$
\item[\rm (iii)] If
$\mu<\infty$ and $\nu=\infty$ then, as $n\to\infty$,
$$\me L_n\sim {1\over \mu}\int_1^n {\me e^{-y(1-W)}\over y}{\rm d}y.$$
\end{enumerate}

\end{thm}
\begin{proof}
Part (i). Set $s_m={\me W^m\over \me(1-W)^m}$. We will use the
representation
\begin{equation}\label{represent2}
\me L_n=\me s_{Z_n},
\end{equation}
which follows from {\rm (\ref{17})}. The array
$c_{n,m}:=\mmp\{Z_n=m\}$
  verifies  the conditions of Lemma
\ref{Toeplitz1} in Appendix, in particular by the  assumption $\mu=\infty$.
Hence the lemma can be applied to  $t_n=\me L_n$,
whence the assertion.
When $\gamma_0$ is well defined the
proof is  simpler, as
in this case the statement follows from
{\rm (\ref{represent2})}, divergence of $Z_n$, and by using
dominated convergence in the case $\gamma_0<\infty$, respectively using Fatou's lemma in the case $\gamma_0=\infty$.

See \cite{GINR} and \cite{GIR} for (ii).

For part (iii)  we use the poissonised version of the Bernoulli sieve, in
which balls are thrown  one-by-one at the epochs of a unit Poisson
process $(\Pi_t)_{t\geq 0}$, independent of $W_k$'s.  One can check that
%
$$\me
(L_{\Pi_t}|(W_k)_{k\in\naturals})=\sum_{k=1}^\infty\left(e^{-tW_1\cdot
\ldots \cdot W_{k-1}(1-W_k)}-e^{-tW_1\cdot \ldots \cdot
W_{k-1}}\right).$$
Recalling definitions
{\rm (\ref{add_rand_walk_def})},{\rm (\ref{N_def})} and setting
$\varphi(t):=\me e^{-t(1-W)}$,
$U(x):=\me
N_x=\sum_{k=1}^{\infty}\mmp\{S_{k-1}\leq x\}$, we have
\begin{eqnarray}
\me L_{\Pi_t}&=&\me \sum_{k=1}^\infty
\bigg(\varphi(te^{-S_{k-1}})-\exp(-te^{-S_{k-1}})\bigg)\nonumber\\
&=& \int_0^\infty
\bigg(\varphi(te^{-x})-\exp(-te^{-x})\bigg)U({\rm d}x)\label{eq2}\\
&=& \sum_{k=1}^\infty (-1)^{k+1}{t^k\over k!}{(1-\me(1-W)^k)\over
(1-\me W^k)},\label{eq1}
\end{eqnarray}
where the familiar formula for Laplace transform of the potential measure,
$$\int_0^\infty e^{-sx}U({\rm d}x)={1\over
1-\me W^s}, ~~~ s>0,$$
has been utilised.
Note that {\rm (\ref{eq1})} is an obvious counterpart of {\rm (\ref{empty_meam_expl})}.

Set $K(t)=\varphi(e^t)-\exp(-e^t)$, $t\in\mr$. Since $\nu=\infty$
and
$$\int_0^\infty {e^{-z(1-W)}-e^{-z}\over z}{\rm d}z=|\log(1-W)|,$$
we conclude that
\begin{equation}\label{aux}
\lim_{t\to\infty}\int_{-\infty}^t K(z){\rm d}z=\infty.
\end{equation}
Applying a minor extension of \cite[Theorem 5]{Sgibnev} to the equality
\begin{equation}\label{key}
\me L_{\Pi_{e^t}}=\int_0^\infty K(t-x)U({\rm d}x),
\end{equation}
which is equivalent to {\rm (\ref{eq2})}, yields

%
%
\begin{eqnarray*}
\me L_{\Pi_{e^t}}\sim {1\over \mu}\int_{0}^t \varphi(e^x){\rm
d}x\sim{1\over \mu}\int_1^{e^t} {\varphi(x)\over x}{\rm d}x.
\end{eqnarray*}

The asymptotics of $\me L_n$ is now obtained by
the depoissonisation Lemma \ref{depoisson} in Appendix.
The lemma
is applicable because  $\me L_{\Pi(t)}$ is slowly varying.
Indeed, slow variation of $\int_1^t \varphi(u){\rm d}u/u$
is checked straightforwardly from $\varphi(t)\downarrow 0$ and the divergence of the integral for $t=\infty$.
\end{proof}

Similarly to the above, the proof of the next theorem is based  on the poissonisation technique.

\begin{thm}{\rm \cite{GINR}}\label{thethe}
If $\mu<\infty$ and $\nu<\infty$ then
$L_n\to_d L_{\infty}$ as $n\to\infty$ for some random variable $L_{\infty}$  whose  distribution  satisfies
$$
\mmp\{L_{\infty}\geq i\}=\frac{1}{\mu}\sum_{j=1}^{\infty}\frac{\me W^j}{j}\mmp\{L_j=i\},\;\;i\in\naturals.
$$
Moreover, the convergence of all moments holds, i.e. $\me L_n^k\to \me L^k_{\infty}<\infty$ for  $k\in\naturals$.
\end{thm}
It is also known that if $\mu<\infty$ and $\nu=\infty$ then
$L_{n}\to_d\infty$ (see \cite{GIR}), and that $L_n\to_P 0$ if
$\nu<\infty$ and $\mu=\infty$. In  the cases not covered by these
results the question about the weak convergence of $L_n$ is open.

Note that Theorem \ref{thethe} only gives implicit specification of the limit law through distributions of $L_n$'s,
which are not easy to determine, with one remarkable exception.
Obviously from the recursive construction of the Bernoulli sieve, the distribution of   $L_1$  is geometric
with parameter $\me W$.
Curiously, the same is true for all $n$ provided the
law of $W$ is symmetric about the midpoint $1/2$.
\begin{assertion}\label{empty_sym}
If  $W=_d 1-W$ then $L_n$ is geometrically distributed with
parameter $1/2$ for all $n\in\naturals$.
\end{assertion}
\begin{proof}
The argument is based on the recursion {\rm (\ref{empty})} for marginal
distributions of the $L_n$'s.
The symmetry $W=_d 1-W$ yields $\me
W^k=\me (1-W)^k$ for all $k\in\naturals$ and
\begin{equation}\label{eq_n_0}
\mmp\{Q^*_n(1)=n\}=\mmp\{Q^*_n(1)=0\}
\end{equation}
for all $n\in\naturals$. We will show by induction on
$n$ that
$\mmp\{L_n=k\}=2^{-k-1}$ for all $k\in\naturalso$.
Using {\rm (\ref{empty})}  and {\rm (\ref{eq_n_0})}
we obtain
\begin{eqnarray*}
\mmp\{L_n=0\} & = &\mmp\{Q^*_n(1)=0\}+\sum_{k=1}^{n-1}\mmp\{L_k=0\}\mmp\{Q^*_n(1)=k\}\\
              & = &\mmp\{Q^*_n(1)=0\}+\frac{1}{2}\Big(1-2\mmp\{Q^*_n(1)=0\}\Big)=\frac{1}{2},
\end{eqnarray*}
by the  induction hypothesis. Assuming now that $\mmp\{L_n=i\}=2^{-i-1}$ for all $i<k$ we have
\begin{eqnarray*}
\mmp\{L_n=k\} & = & \sum_{j=1}^{n-1}\mmp\{Q^*_n(1)=j\}\mmp\{L_j=k\}+\mmp\{Q^*_n(1)=n\}\mmp\{L_n=k-1\}\\
                      & = &  2^{-k-1}\Big(1-2\mmp\{Q^*_n(1)=0\}\Big)+\mmp\{Q^*_n(1)=0\}2^{-k}=2^{-k-1},
\end{eqnarray*}
and the proof is complete.
\end{proof}

Alternatively, one can use  a representation of $L_n$ through the sojourns
of the Markov chain $Q^*_n$ in positive states.
Indeed,  recall that $L_1$ has geometric distribution with parameter $\me W$.
Then using
{\rm (\ref{eq_n_0})}
and induction it can be checked that the distribution of $L_n$ does not depend on $n\geq 1$.

\section{Appendix.}


For ease of reference we include a result
due to Toeplitz and Schur (see  \cite{DivSer}, Theorem 2 on p.~ 43 and Theorem 9 on p.~
52). We rewrite it in a  form suitable for our
purposes.
\begin{lemma}\label{Toeplitz1}
Let $\{s_n,n\in\mathbb{N}\}$ be any sequence of real numbers
and let $\{c_{nm},~n,m\in\naturals\}$ be a nonnegative array.
Define another sequence $\{t_n,n\in\mathbb{N}\}$ by
$t_n=\sum_{m=1}^{n}c_{nm}s_m$. If
\begin{itemize}
\item[\rm(i)] $\lim_{n\to\infty}c_{nm}=0$ for all $m$,
\item[\rm(ii)] $\,\lim_{n\to\infty}\sum_{m=1}^{n}c_{nm}=1$,
\end{itemize}
then
$$\liminf_{n\to\infty}{s_n}\leq\liminf_{n\to\infty}{t_n}\leq\limsup_{n\to\infty}{t_n}\leq\limsup_{n\to\infty}{s_n}\leq +\infty.$$
\end{lemma}


Now we  address the issue of depoissonisation.

\begin{lemma}\label{depoisson}
If $\,\me L_{\Pi_t}$ is slowly varying and $\liti \me L_{\Pi_t}\in(0,\infty]$ then
$$\me L_{n}\sim \me L_{\Pi_n}, ~~{\rm as~~}n\to\infty.$$
\end{lemma}
\begin{proof}
For any fixed $\varepsilon\in (0,1)$,
\begin{eqnarray*}
\me L_{\Pi_t}&=&\me L_{\Pi_t}1_{\{|\Pi_t-t|> \epsilon t\}}+
\me L_{\Pi_t}1_{\{|\Pi_t-t|\leq \epsilon t\}} =: A(t)+B(t).
\end{eqnarray*}
Sublinearity of $\me L_{\Pi_t}$
and the elementary large deviation bound for the Poisson distribution \cite{Bah},
\begin{equation}\label{ba}
\mmp\{|\Pi_t-t|>\epsilon t\}<
c_1e^{-c_2t}, \ t>0
\end{equation}
with some $c_1, c_2>0$, yield $A(t)\to 0$.




It remains to evaluate $B(t)$. Since both $M_n$ and $K_n$ are
non-decreasing, we have
\begin{eqnarray*}
B(t)=\me (M_{\Pi_t}-K_{\Pi_t})1_{\{|\Pi_t-t|\leq \varepsilon t\}}\leq \me L_{[(1-\varepsilon)t]}+\me (M_{[(1+\varepsilon)t]}-M_{[(1-\varepsilon)t]}).
\end{eqnarray*}
Similarly, $B(t)\geq \me L_{[(1+\varepsilon)t]}-\me(M_{[(1+\varepsilon)t]}-M_{[(1-\varepsilon)t]}).$
To proceed, we need an auxiliary lemma.

\begin{lemma}\label{cruc}
$\lim_{\varepsilon\downarrow 0} \liti \me
(M_{[(1+\varepsilon)t]}-M_{[(1-\varepsilon)t]})=0$.
\end{lemma}
\begin{proof}
Let $\{V_i:i\in\naturals\}$ be the sequence of independent exponentially distributed random
variables with $\me V_i=1/i$. The sequence $G_0:=0,
(G_n:=V_1+\ldots+V_n-h(n))_{n\in\naturals}$ is an $L_2$-bounded
martingale with respect to the natural filtration,
where $h(n)=\sum_{j=1}^n 1/j$.
 This implies
that $G_n$ converges almost surely to some random variable $G$,
with $\me G=0$. By Doob's inequality
\begin{equation}\label{do}
\me (\sup_{n\in\naturalso} |G_n|)^2<\infty.
\end{equation}

Recalling the notation $T_n=\max(E_1,\dots,E_n)$,
\begin{eqnarray*}
D(t)&:=& \me \bigg(M_{[(1+\varepsilon)t]}-M_{[(1-\varepsilon)t]}\bigg)= \me \bigg(U(T_{[(1+\varepsilon)t]})-U(T_{[(1-\varepsilon)t]})\bigg)\\
&=& \me \bigg(U(h([(1+\varepsilon) t])+G_{[(1+\varepsilon )t]})-U(h([(1+\varepsilon) t]))\bigg)\\
&+& \bigg(U(h([(1+\varepsilon )t]))-U(h([(1-\varepsilon) t]))\bigg)\\
&-& \me \bigg(U(h([(1-\varepsilon )t])+G_{[(1-\varepsilon )t]})-U(h([(1-\varepsilon)t]))\bigg)
=:\me D_1(t)+D_2(t,\varepsilon)-\me D_3(t).
\end{eqnarray*}
Recall that Blackwell's renewal theorem states that $$\lim_{x\to\infty}
\bigg(U(x+y)-U(x)\bigg)={y\over \mu},$$
where the convergence is locally uniform in $y\in\mr$.
From this we conclude that $\liti
D_2(t,\varepsilon)=\log{1+\varepsilon\over 1-\varepsilon}$. Hence,
$$\lim_{\varepsilon\downarrow 0} \liti D_2(t,\varepsilon)=0.$$
By the same argument
$\liti D_1(t)={G/\mu}$
almost surely.
Now we want to show that $U(\sup_{n\in\naturalso} |G_n|)$ is an
integrable majorant for $|D_1(t)|$. To this end, we use
subadditivity and monotonicity of $U$:
\begin{eqnarray*}
|D_1(t)|&=&D_1(t)1_{\{G_{[(1+\varepsilon)t]}\geq
0\}}+(-D_1(t))1_{\{G_{[(1+\varepsilon)t]}<0\}}\\
&\leq& U(G_{[(1+\varepsilon )t]})1_{\{G_{[(1+\varepsilon)t]}\geq
0\}}+U(-G_{[(1+\varepsilon)t]})1_{\{G_{[(1+\varepsilon)t]}< 0\}}= U(|G_{[(1+\varepsilon)t]}|)\leq U(\sup_{n\in\naturalso} |G_n|).
\end{eqnarray*}

An appeal to estimate $U(x)<ax+b$ (with some $a,b>0$) and {\rm (\ref{do})} allows us to conclude that
$$\me U(\sup_{n\in\naturalso}|G_n|)\leq a\me \sup_{n\in\naturalso} |G_n|+b<\infty.$$ Consequently,
invoking
the dominated convergence we show  that
$\liti \me D_1(t)=0$. Along the same lines
 $\liti \me
D_3(t)=0$ is shown. The proof is complete.
\end{proof}

We are ready to finish the proof.
Assume first that $\liti \me L_{\Pi_t}=c\in(0,\infty)$. Letting $n\to\infty$  then $\varepsilon\to 0$
in the inequality
\begin{equation}\label{5}
\me L_{\Pi_{n/(1-\varepsilon)}}\leq
A(n/(1-\varepsilon))+ \me L_n+\me
(M_{[{1+\varepsilon\over 1-\varepsilon}n]}-M_n),
\end{equation}
we obtain $\lim\inf_{n\to\infty}\,\me L_n\geq
c$. The upper bound follows in the same way from the
inequality
\begin{equation}\label{678}
\me L_{\Pi_{n/(1+\varepsilon)}}\geq \me L_n-\me
(M_n-M_{[{1-\varepsilon\over 1+\varepsilon}n]}).
\end{equation}

In the case $\liti \me L_{\Pi_t}=\infty$, divide inequalities
{\rm (\ref{5})} and {\rm (\ref{678})} by $\me L_{\Pi_n}$ and let $n$ go to
$\infty$ keeping in mind that by slow variation $\lin {\me L_{\Pi_{\delta n}}\over
\me L_{\Pi_n}}=1$ for every $\delta>0$.
\end{proof}

\nocite{*}
\bibliographystyle{abbrv}
\bibliography{bibliography}
\end{document}